\documentclass[12pt,draft]{article}
\usepackage{amsmath,amscd}
\usepackage{amssymb,latexsym,amsthm}
\usepackage{color}
\usepackage{amssymb}
\usepackage{color}
\usepackage{amsmath,amsthm,amscd}
\usepackage[latin1]{inputenc}
\usepackage{lscape}
\usepackage{fancyhdr}
\usepackage{amsfonts}
\usepackage{pb-diagram}

\numberwithin{equation}{section}
\newtheorem{theorem}{Theorem}[section]

\newtheorem{proposition}[theorem]{Proposition}
\newtheorem{lemma}[theorem]{Lemma}

\newtheorem{corollary}[theorem]{Corollary}

\theoremstyle{definition}
\newtheorem{example}[theorem]{Example}
\newtheorem{definition}[theorem]{Definition}

\newtheorem{remark}[theorem]{Remark}

\title{\textbf{Double Ore extensions versus graded skew PBW extensions}}
\author{James Yair G\'OMEZ, H\'ector SU\'AREZ\\School of Mathematics and Statistics,
Faculty of Science,\\ Pedagogical and Technological University of
Colombia,\\ Tunja, Colombia}
\date{}
\begin{document}
\maketitle
\begin{abstract}
\noindent In this paper we present necessary and sufficient
conditions for a graded (trimmed) double Ore extension to be a
graded (quasi-commutative) skew PBW extension. Using this fact, we
prove that a graded skew PBW extension $A = \sigma(R)\langle x_1,x_2
\rangle$ of an Artin-Schelter regular algebra $R$ is Artin-Schelter
regular. As a consequence,  every graded skew PBW extension $A =
\sigma(R)\langle x_1,x_2 \rangle$ of a connected skew Calabi-Yau
algebra $R$ of dimension $d$ is skew Calabi-Yau of dimension $d+2$.

\bigskip


\bigskip

\end{abstract}

\section{Introduction}
Regular algebras were defined by Artin and Schelter in \cite{ejm2}
and they are now known in the literature as Artin-Schelter regular
algebras. Artin-Schelter regular algebras of global dimension two
and global dimension three were classified in \cite{ejm2}, but there
are also many open questions about these algebras; for example, the
classification of Artin-Schelter regular algebras of global
dimension greater than three. Different authors have focused on
studying Artin-Schelter regular algebras, especially those of global
dimension four and global dimension five (see for example
\cite{Wang,Zhang,Zhang2,Zhun}). One of the main objectives in
noncommutative algebraic geometry is the classification of
noncommutative projective spaces. Quadratic Artin-Schelter regular
algebras  can be thought of as the coordinate rings of
noncommutative projective spaces (see \cite{Rogalski}). Calabi-Yau
algebras were defined by Ginzburg in \cite{Ginzburg1} and the skew
Calabi-Yau algebras were defined as a generalization of them. Reyes,
Rogalski and
 Zhang in \cite{Reyes} proved that for connected algebras, the skew
Calabi-Yau property is equivalent  to the Artin-Schelter regular
property.  Calabi-Yau algebras are skew Calabi-Yau, but some skew
Calabi-Yau algebras are not Calabi-Yau (for example the Jordan
 plane).\

Zhang and Zhang in \cite{Zhang} introduced a special class of
algebras called double Ore extensions. They proved that a connected
graded double Ore extension of an Artin-Schelter regular algebra is
Artin-Schelter regular. The same authors  in \cite{Zhang2}
constructed 26 families of Artin-Schelter regular algebras of global
dimension four, using double Ore extensions. Carvalho and Matczuk in
\cite{Pau} described  double Ore extensions that can be presented as
iterated Ore extensions. Other properties of double Ore extensions
have been studied (see for example \cite{Zhun}).\

Skew PBW extensions were defined in \cite{Galleg}, and recently
studied in many papers (see for example
\cite{LePai,ReyesSuarez2018-1,ReyesSuarez2018-2,ReyesSuarez2018-3}).
The second author in \cite{Hec} defined the graded skew PBW
extensions for an algebra $R$. This definition generalizes graded
iterated Ore extensions. Some properties of graded skew PBW
extensions have been studied in \cite{Hec,Hect}. The Artin-Schelter
regular property and the skew Calabi-Yau  condition for graded skew
PBW extensions were studied in \cite{Hect}; they proved that every
graded quasi-commutative skew PBW extension of an Artin-Schelter
regular algebra is Artin-Schelter regular algebra, every graded
quasi-commutative skew PBW extension of a connected skew Calabi-Yau
algebra is skew Calabi-Yau, and graded skew PBW extensions of
connected Auslander regular algebras are Artin-Schelter regular and
skew Calabi-Yau.\

In the current literature, as far as we are aware, there is no study
of the relationships between skew PBW extensions and double Ore
extensions. In this paper we study some relations between (graded)
right double Ore extensions and (graded) skew PBW extensions, and
between (graded) trimmed right double Ore extensions and (graded)
quasi-commutative skew PBW extensions. We conclude that  a graded
skew PBW extension $A = \sigma(R)\langle x_1,x_2 \rangle$ of an
Artin-Schelter regular algebra $R$ is Artin-Schelter regular, and a
graded skew PBW extension $A = \sigma(R)\langle x_1,x_2 \rangle$ of
a connected skew Calabi-Yau algebra $R$ of dimension $d$ is skew
Calabi-Yau of dimension $d+2$.

The main results of this paper are found in Section
\ref{sec.versus}. In Theorem \ref{rel11} we give conditions for a
double Ore extension to be a skew PBW extension. In Theorem
\ref{fff} we show that  a connected graded skew PBW extension $A=
\sigma(R)\langle x_1,x_2\rangle$ is a connected graded double Ore
extension of $R$. As a consequence of this theorem we show in
Corollary \ref{cor.traim} that if $A= \sigma(R)\langle
x_1,x_2\rangle$ is a connected graded quasi-commutative skew PBW
extension of $R$, then $A$ is a connected graded trimmed right
double Ore extension of $R$. In Example \ref{asejeml} we present
examples of connected graded skew PBW extensions that are connected
graded right double Ore extensions. In Theorem \ref{rel112} we give
conditions for a graded right double Ore extension
$A=R_P[x_1,x_2,\sigma,\delta,\tau]$ to be a graded skew PBW
extension of $R$. In Examples \ref{oreejm1} and \ref{oreejm3} we
present examples of connected graded right double Ore extensions
that are connected graded skew PBW extensions. In Theorem
\ref{relaciondoblesgrad} (Theorem \ref{teo.traim}) we give necessary
and sufficient conditions for a connected graded (trimmed right)
double Ore extensions of a connected algebra $R$ to be a graded
(quasi-commutative) skew PBW extension. Also in Theorem \ref{result}
we prove that if  $R$ is an Artin-Schelter regular algebra and $A =
\sigma(R)\langle x_1,x_2 \rangle$ is a graded skew PBW extension of
$R$, then $A$ is  Artin-Schelter regular and $gld(A)=gld(R)+2$. As a
consequence in  Corollary \ref{resultCalab} we conclude that every
graded skew PBW extension $A = \sigma(R)\langle x_1,x_2 \rangle$ of
a connected skew Calabi-Yau algebra $R$ of dimension $d$ is skew
Calabi-Yau of dimension $d+2$. Some examples of Artin-Schelter
regular algebras and skew Calabi-Yau algebras are presented in
Example \ref{ex.ArtinCal}.

\section{Double Ore extensions and skew PBW extensions}\label{General}

From now on, and if not stated otherwise, we  fix the following
notation: $K$ is a field, all algebras are $K$-algebras, vector
spaces are $K$-vector spaces and $\otimes$ is  $\otimes_{K}$. A
graded algebra $B=\bigoplus_{p\geq 0}B_p$ is called \emph{connected}
if $B_0=K$. In \cite[Definition 1.4]{Rogalski}, the concept of
finitely graded algebra was presented. It is said that an algebra
$B$ is \emph{finitely graded} if the following conditions hold:
\begin{enumerate}
\item[\rm (i)] $B$ is $\mathbb{N}$-graded (positively graded): $B = \bigoplus_{j\geq
0}B_j$.
\item[\rm (ii)] $B$ is connected.
\item[\rm (iii)] $B$ is \emph{finitely generated} as algebra, i.e., there is a finite set of
elements $x_1,\dots, x_n\in B$ such that the set
$\{x_{i_1}x_{i_2}\cdots x_{i_m}\mid 1\leq i_j\leq n, m\geq 1\} \cup
\{1\}$ spans $B$ as a vector space.
\end{enumerate}
Let $B$ be a finitely graded algebra. $B$ is \emph{finitely
generated} as $K$-algebra if and only if $B = K\langle x_1,\dots,
x_m\rangle/I$, where $I$ is a proper homogeneous two-sided ideal of
$K\langle x_1,\dots, x_m\rangle$. Moreover, for every $n\in
\mathbb{N}$, $\dim_KB_n < \infty$, i.e., $B$ is locally finite.

\begin{definition}\label{def.Regular}  Let $B = K\oplus B_1\oplus B_2\oplus \cdots$ be a finitely presented  algebra over
 $K$. The algebra $B$ is called \emph{Artin-Schelter regular} if the
following conditions hold:
\begin{enumerate}
\item[\rm (i)] $B$ has finite global dimension $d$.
\item[\rm (ii)] $B$ has finite GK-dimension.
\item[\rm (iii)] $B$ is \emph{Gorenstein}, meaning that $Ext_B^i(K,
B) =0$ if $i \neq d$, and $Ext^d_B(K, B)\cong K(l)$, for some
integer $l$.
\end{enumerate}
\end{definition}

An \emph{Ore extension} of a ring $R$ in the indeterminate $x$  is
the noncommutative polynomial ring $R[x; \sigma, \delta]$ with
product defined by $xr = \sigma(r)x + \delta(r)$, where $\sigma : R
\to R$ is an  endomorphism of $R$ and $\delta$ is a
$\sigma$-derivation of $R$, i.e., $\delta(r+r') =
\delta(r)+\delta(r')$ and $\delta(rr') = \sigma(r)\delta(r') +
\delta(r)r'$, for any $r$, $r'\in R$. The $R$-free basis is $\{x^l|l
\geq 0\}$. Let $B= R[x_1; \sigma_1, \delta_1]$ be an Ore extension
of  $R$ and $C = B[x_2; \sigma_2,\delta_2]$ be an Ore extension of
$B$, then $C$ is a free left $R$-module with a basis
$\{x_1^{n_1}x_2^{n_2}\}_{n_1,n_2\geq0}$. The ring $C$ is called an
iterated Ore extension of $R$. More generally we can consider the
\emph{iterated Ore extension} $R[x_1; \sigma_1, \delta_1] \cdots
[x_n; \sigma_n, \delta_n]$ where $\sigma_i$, $\delta_i$ are defined
on $R[x_1; \sigma_1, \delta_1] \cdots [x_{i-1}; \sigma_{i-1},
\delta_{i-1}]$, i.e.,
$$\sigma_i, \delta_i : R[x_1; \sigma_1, \delta_1]\cdots [x_{i-1}; \sigma_{i-1}, \delta_{i-1}] \to R[x_1; \sigma_1, \delta_1]\cdots [x_{i-1}; \sigma_{i-1}, \delta_{i-1}].$$

A \emph{skew polynomial ring} is an Ore extension
$R[x_1;\sigma_1,\delta_1]\cdots [x_n;\sigma_n,\delta_n]$ satisfying:
$x_ix_j = x_jx_i,\ for \;\; 1 \leq i,j \leq n$ and $\sigma_i(R),
\delta_i(R) \subseteq R$, \ for \ $ 1\leq i \leq n.$

\begin{definition} [\cite{Zhang}, Definition 1.3]\label{defore2}
Let $R$ be an algebra and $B$ be another algebra, containing $R$ as
a subring.
\begin{itemize}
\item[(i)] $B$ is a \emph{right double Ore extension} of $R$ if
the following conditions hold:
\begin{enumerate}
\item[\rm (ia)] $B$ is generated by $R$ and two new variables $x_1$ and $x_2$.
\item[\rm (iia)] The variables $\{x_1, x_2\}$ satisfy the relation
\begin{equation}\label{R1}
x_2x_1 = p_{12}x_1x_2 + p_{11}x_1^2 + \tau_1x_1 + \tau_2x_2 +
\tau_0,
\end{equation}
where $p_{12}, p_{11}\in K$ and $\tau_1, \tau_2, \tau_0 \in R$.
\item[\rm (iiia)] As a left $R$-module, $B =\sum\limits_{\alpha_1,\alpha_2\geq 0}Rx_1^{\alpha_1}x_2^{\alpha_2}$ and it is a left free $R$-module with $\{x_1^{\alpha_1}x_2^{\alpha_2} \mid\alpha_1\geq 0,\alpha_2\geq 0\}$.
\item[\rm (iva)] \begin{equation}\label{4ore2} x_1R + x_2R \subseteq Rx_1 + Rx_2 + R.\end{equation}
\end{enumerate}
\item[(ii)] $B$ is a \emph{left double Ore extension} of $R$ if
the following conditions hold:
\begin{enumerate}
\item[\rm (ib)] $B$ is generated by $R$ and two new variables $x_1$ and $x_2$.
\item[\rm (iib)] The variables $\{x_1, x_2\}$ satisfy the relation
\begin{equation}\label{L1}
x_1x_2 = p'_{12}x_2x_1 + p'_{11}x_1^2 + x_1\tau'_1 + x_2\tau'_2 +
\tau'_0,
\end{equation}
where $p'_{12}, p'_{11}\in K$ and $\tau'_1, \tau'_2, \tau'_0 \in R$.
\item[\rm (iiib)] As a right $R$-module, $B =\sum\limits_{\alpha_1,\alpha_2\geq 0}Rx_1^{\alpha_1}x_2^{\alpha_2}$ and it is a right free $R$-module with $\{x_1^{\alpha_1}x_2^{\alpha_2} \mid\alpha_1\geq 0,\alpha_2\geq 0\}$.
\item[\rm (ivb)] $x_1R + x_2R \subseteq Rx_1 + Rx_2 + R$.
\end{enumerate}
\item[(iii)] $B$ is a \emph{double Ore extension} if it is  left and  right double Ore extension of $R$ with the same
generating set $\{x_1,x_2\}$.
\item[(iv)] $B$ is a \emph{graded right (left) double Ore extension}
if all relations of $B$ are homogeneous with assignment
$deg(x_1)=deg(x_2) =1$.
\end{itemize}
\end{definition}

Let $P$ denote the set of scalar parameters $\{p_{12},p_{11}\}$ and
let $\tau$ denote the set $\{\tau_1,\tau_2,\tau_0 \}$. We call $P$
the parameter and $\tau$ the tail (see \cite[Page 2671]{Zhang}).
Using the ideas given in \cite[Page 2671]{Zhang}, we rewrite the
condition (\ref{4ore2}) as follows:
\begin{equation}\label{R2}
\left(\begin{smallmatrix}
x_1 \\
x_2 \\
\end{smallmatrix}\right)r
 := \left(\begin{smallmatrix}
x_1r \\
x_2r \\
\end{smallmatrix}\right)= \left(\begin{smallmatrix}
\sigma_{11}(r) && \sigma_{12}(r) \\
\sigma_{21}(r) && \sigma_{22}(r) \\
\end{smallmatrix}\right)\left(\begin{smallmatrix}
x_1 \\
x_2 \\
\end{smallmatrix}\right) + \left(\begin{smallmatrix}
\delta_1(r) \\
\delta_2(r) \\
\end{smallmatrix}\right),
\end{equation}

for all $r \in R$.\\

Writing $$\sigma(r) = \left(\begin{smallmatrix}
\sigma_{11}(r) && \sigma_{12}(r) \\
\sigma_{21}(r) && \sigma_{22}(r) \\
\end{smallmatrix}\right) \quad\text{and}\quad \delta(r) = \left(\begin{smallmatrix}
\delta_{1}(r)  \\
\delta_{2}(r)  \\
\end{smallmatrix}\right),$$  $\sigma$ is a $K$-linear map from $R$ to $M_2(R)$, where $M_2(R)$ is the $2\times2$-matrix algebra over  $R$, and
$\delta$ is a $K$-linear map from $R$ to the column $R$-module
$R^{\oplus{2}}:= \left(\begin{smallmatrix}
R \\
R  \\
\end{smallmatrix}\right)$. Indeed, as $A$ satisfies the relation $x_i(r+s)=x_ir+x_is$, for each $r,s \in R$ and $i=1,2$,
then:

\begin{equation}\label{eq.I}
\begin{array}{lcl} x_1(r+s) &=& \sigma_{11}(r+s)x_1+
\sigma_{12}(r+s)x_2 + \delta_1(r+s),
\end{array}
\end{equation}

\begin{equation}\label{eq.II}
\begin{array}{lcl}
x_1r+x_1s &=& \sigma_{11}(r)x_1+ \sigma_{12}(r)x_2 + \delta_1(r) + \sigma_{11}(s)x_1+ \sigma_{12}(s)x_2 + \delta_1(s)\\
          &=&(\sigma_{11}(r)+\sigma_{11}(s))x_1+ (\sigma_{12}(r)+\sigma_{12}(s))x_2 +
          \delta_1(r)+\delta_1(s),
\end{array}
\end{equation}
\begin{equation}\label{eq.III}
\begin{array}{lcl} x_2(r+s) &=& \sigma_{21}(r+s)x_1+
\sigma_{22}(r+s)x_2 + \delta_2(r+s),
\end{array}
\end{equation}
\begin{equation}\label{eq.IV}
\begin{array}{lcl}
x_2r+x_2s &=& \sigma_{21}(r)x_1+ \sigma_{22}(r)x_2 + \delta_2(r) + \sigma_{21}(s)x_1+ \sigma_{22}(s)x_2 + \delta_2(s)\\
          &=&(\sigma_{21}(r)+\sigma_{21}(s))x_1+ (\sigma_{22}(r)+\sigma_{22}(s))x_2 +
          \delta_2(r)+\delta_2(s).
\end{array}
\end{equation}

If we equal (\ref{eq.I}) with  (\ref{eq.II}) and (\ref{eq.III}) with (\ref{eq.IV}), we have\\

 $ \begin{array}{llcll}
\sigma(r+s) &= & \left(\begin{matrix}
\sigma_{11}(r+s) && \sigma_{12}(r+s) \\
\sigma_{21}(r+s) && \sigma_{22}(r+s) \\
\end{matrix}\right)\\&=&  \left(\begin{matrix}
\sigma_{11}(r)+\sigma_{11}(s) && \sigma_{12}(r)+\sigma_{12}(s) \\
\sigma_{21}(r)+\sigma_{21}(s) && \sigma_{22}(r)+\sigma_{22}(s) \\
\end{matrix}\right)\\
&&&&\\
& = & \left(\begin{matrix}
\sigma_{11}(r) && \sigma_{12}(r) \\
\sigma_{21}(r) && \sigma_{22}(r)\\
\end{matrix}\right)
+  \left(\begin{matrix}
\sigma_{11}(s) && \sigma_{12}(s) \\
\sigma_{21}(s) && \sigma_{22}(s) \\
\end{matrix}\right)\\
&&&&\\
 &=& \sigma(r)+\sigma(s).
\end{array} $
\vspace{0.3cm}

 $ \begin{array}{llllll}
\delta(r+s) &=& \left(\begin{matrix}
\delta_{1}(r+s)  \\
\delta_{2}(r+s)  \\
\end{matrix}\right) &=& \left(\begin{matrix}
\delta_{1}(r) +\delta_1(s) \\
\delta_{2}(r) + \delta_2(s) \\
\end{matrix}\right)= &  \left(\begin{matrix}
\delta_{1}(r)  \\
\delta_{2}(r)  \\
\end{matrix}\right) +  \left(\begin{matrix}
\delta_{1}(s)  \\
\delta_{2}(s)  \\
\end{matrix}\right)\\
&&&&&\\
 &&&&& =\delta(r)+\delta(s).
\end{array}$
\vspace{0.3cm}

Furthermore, as for  $k \in K$, $x_i(kr) = k(x_ir)$, with $i=1,2$ then\\

 $ \begin{array}{llcll}
\sigma(kr) &= & \left(\begin{matrix}
\sigma_{11}(kr) && \sigma_{12}(kr) \\
\sigma_{21}(kr) && \sigma_{22}(kr) \\
\end{matrix}\right)&=&  \left(\begin{matrix}
k \sigma_{11}(r) && k\sigma_{12}(r) \\
k \sigma_{21}(r) && k\sigma_{22}(r) \\
\end{matrix}\right)\\
&&&&\\
& = & k \left(\begin{matrix}
\sigma_{11}(r) && \sigma_{12}(r) \\
\sigma_{21}(r) && \sigma_{22}(r)\\
\end{matrix}\right)
& = &  k\sigma(r).
\end{array} $
\vspace{1cm}

$ \begin{array}{lllll} \delta(kr) &=& \left(\begin{matrix}
k\delta_{1}(r)  \\
k\delta_{2}(r)  \\
\end{matrix}\right) &=& k\left(\begin{matrix}
\delta_{1}(r)  \\
\delta_{2}(r) \\
\end{matrix}\right)= k \delta(r).
\end{array}$\\

 Thus,  equation (\ref{R2}) can also be written as
$$\left(\begin{smallmatrix}
x_1\\
x_2 \\
\end{smallmatrix}\right)r = \sigma(r)\left(\begin{smallmatrix}
x_1\\
x_2\\
\end{smallmatrix}\right) + \delta(r).$$

This equation is a natural generalization of the product in an Ore extension.\\

\begin{definition} Let $\sigma : R \rightarrow M_2(R)$ be an algebra
homomorphism. A $K$-linear map $\delta: R \mapsto R^{\oplus{2}}$ is
called a $\sigma$-derivation if $\delta(rs)=\sigma(r)\delta(s) +
\delta(r)s$ for all $r,s \in R$.
\end{definition}

\begin{remark} Let $A$ be a right double Ore extension of $R$. With the above
notation, $\sigma$ and $\delta$ are uniquely determined, $\sigma$ is
a $K$-linear map from $R$ to $M_2(R)$, and $\delta$ is a $K$-linear
map from $R$ to $R^{\oplus{2}}$. Together with Definition
\ref{defore2}-(i), all symbols of $\{P, \sigma, \delta, \tau\}$ are
defined now. When everything is understood, a right double extension
or a double Ore extension $A$ is also denoted by $A= R_P[x_1,x_2;
\sigma, \delta,\tau]$. By this notation, we are working with a right
double Ore extension though $A$ can be also a (left) double Ore
extension (see \cite[Convention 1.6]{Zhang}).
\end{remark}

\begin{definition} Let $A=R_P[x_1,x_2; \sigma, \delta,\tau]$ be a a right double Ore
extension.  If $\delta$ is a zero map and $\tau$ consists of zero
elements, then the right double Ore extension is denoted by
$R_P[x_1,x_2;\sigma]$ and is called a \emph{trimmed right double Ore
extension}.
\end{definition}

\begin{lemma}[\cite{Zhang}, Lemma 1.7]\label{lema1.7}
Let $A = R_P[x_1,x_2;\sigma, \delta, \tau]$ be a right double
extension of $R$. With $\{ \sigma,\delta \}$  defined as in
(\ref{R2}). Then the following holds:
\begin{enumerate}
\item[(i)] $\sigma: R \rightarrow M_2(R)$ is an algebra homomorphism.
\item[(ii)] $\delta: R \rightarrow R^{\oplus{2}}$ is a $\sigma$-derivation.
\item[(iii)] If $\sigma : R \rightarrow M_2(R)$ is an algebra homomorphism and  $\delta: R \rightarrow R^{\oplus{2}}$ is a $\sigma$-derivation, then (\ref{R2})
holds for all elements $r \in R$ if and only if it holds for a set
of generators.
\end{enumerate}
\end{lemma}

By the  Lemma \ref{lema1.7}, when $A = R_P[x_1, x_2; \sigma, \delta,
\tau]$ is a right double Ore extension, then $\sigma$ is an algebra
homomorphism and $\delta$ a $\sigma$-derivation. We now call
$\sigma$ a homomorphism, $\delta$ a derivation.\\

Suppose $A= R_P[x_1,x_2;\sigma,\delta,\tau]$ is a right double
extension. Zhang and Zhang in \cite{Zhang} defined the
$P$-determinant of $\sigma := \left(\begin{smallmatrix}
\sigma_{11} && \sigma_{12} \\
\sigma_{21} && \sigma_{22} \\
\end{smallmatrix}\right)$ or just the determinant of $\sigma$ to
be the map $$\det\sigma(r) =
-p_{11}\sigma_{12}\sigma_{11}(r)+\sigma_{22}\sigma_{11}(r)-p_{12}\sigma_{12}\sigma_{21}(r),$$
for all $r\in R$. This is a $K$-linear map from $R$ to itself.

\begin{remark}\label{detll} Let $A= R_P[x_1,x_2;\sigma,\delta,\tau]$ be a right double
extension. Then the following holds:
\begin{itemize}
\item[(i)] If $P=\{1,0\}$, then $\det\sigma = \sigma_{22}\sigma_{11} - \sigma_{12}\sigma_{21}.$
\item[(ii)] If $p_{12} \neq 0$,  then  $\det\sigma = -p_{12}^{-1}p_{11}\sigma_{11}\sigma_{22} -
p_{12}^{-1}\sigma_{21}\sigma_{12}+ \sigma_{11}\sigma_{22}$ (see
\cite[Page 2677]{Zhang}).
\end{itemize}
\end{remark}

\begin{proposition}[\cite{Zhang}, Proposition 2.1]\label{prop2.1do} Let  $A= R_P[x_1,x_2;\sigma,\delta,\tau]$ be a right double
extension.
\begin{itemize}
\item[(i)] $\det\sigma$ is
an algebra endomorphism of $R$.
\item[(ii)] If $\sigma$ is invertible, then $\det\sigma$ is
invertible.
\item[(iii)] Suppose $p_{12}\neq0$. If $\det\sigma$ is invertible, then $\sigma$ has
a right inverse. Furthermore, if $A$ is a connected graded right
double extension, then $\sigma$ is invertible and $A$ is a double
extension.
\end{itemize}
\end{proposition}

Artin-Schelter regular property was studied in \cite[Theorem 3.3 and
Theorem 0.2]{Zhang}, for connected graded  double Ore extensions and
connected graded trimmed double Ore extensions.

\begin{theorem}[\cite{Zhang}, Theorem 3.3]\label{Teorema3.3}
Let $R$ be an Artin-Schelter regular algebra and let $A$ be a
connected graded trimmed double Ore extension $R_P[x_1,
x_2;\sigma]$. Then $A$ is Artin-Schelter regular and $gld(A) =gld(R)
+ 2$.
\end{theorem}

\begin{theorem}[\cite{Zhang}, Theorem 0.2]\label{Teorema0.2}
Let $R$ be an Artin-Schelter regular algebra. If $A$ is a connected
graded double Ore extension of $R$, then $A$ is Artin-Schelter
regular and $gld(A)=gld(R)+2$.
\end{theorem}

The skew PBW extensions introduced in \cite{Galleg} include many
algebras of interest for modern mathematical physicists. Skew PBW
extensions are defined by a ring and a set of variables with
relations between them.

\begin{definition}[\cite{Galleg}, Definition 1]\label{defpbwt}
Let $R$ and $A$ be rings. We say that $A$ is a \textit{skew PBW
extension of} $R$ if the following conditions hold:
\begin{enumerate}
\item[\rm (i)]$R\subseteq A$;
\item[\rm (ii)] there exist finitely many elements $x_1,\dots ,x_n\in A$ such that $A$ is a left free $R$-module, with basis the
set of standard monomials
\begin{center}
${\rm Mon}(A):= \{x^{\alpha}:=x_1^{\alpha_1}\cdots
x_n^{\alpha_n}\mid \alpha=(\alpha_1,\dots ,\alpha_n)\in
\mathbb{N}^n\}$.
\end{center}
Moreover, $x^0_1\cdots x^0_n := 1 \in {\rm Mon}(A)$.

\item[\rm (iii)]For each $1\leq i\leq n$ and any $r\in R\ \backslash\ \{0\}$, there exists an element $c_{i,r}\in R\ \backslash\ \{0\}$ such that
\begin{equation}\label{torcida3}
x_ir-c_{i,r}x_i\in R.
\end{equation}
\item[\rm (iv)]For $1\leq i,j\leq n$ there exists $c_{i,j}\in R\ \backslash\ \{0\}$ such that
\begin{equation}\label{torcida4}
x_jx_i-c_{i,j}x_ix_j\in R+Rx_1+\cdots +Rx_n.
\end{equation}
Under these conditions we will write $A=\sigma(R)\langle
x_1,\dots,x_n\rangle$. For $\alpha=(\alpha_1,\dots,\alpha_n)\in
\mathbb{N}^n$, $\sigma^{\alpha}:=\sigma_1^{\alpha_1}\cdots
\sigma_n^{\alpha_n}$, $|\alpha|:=\alpha_1+\cdots+\alpha_n$.
\end{enumerate}
\end{definition}

\begin{proposition}[\cite{Galleg}, Proposition 3]\label{prop23}
Let $A$ be a skew PBW extension of $R$. For each $1\leq i\leq n$,
there exists an injective endomorphism $\sigma_i:R\rightarrow R$ and
a $\sigma_i$-derivation $\delta_i:R\rightarrow R$ such that
\begin{equation}
x_ir=\sigma_i(r)x_i+\delta_i(r),\ \ \ \  \ r \in R.
\end{equation}
\end{proposition}

Next we present some important subclasses of skew PBW extensions.

\begin{definition}\label{sigmapbwderivationtype}
Let $A$ be a skew PBW extension of $R$, $\Sigma:=\{\sigma_1,\dotsc,
\sigma_n\}$ and $\Delta:=\{\delta_1,\dotsc, \delta_n\}$, where
$\sigma_i$ and $\delta_i$ ($1\leq i\leq n$) are as in Proposition
\ref{prop23}.
\begin{enumerate}
\item[\rm (a)]\label{def.quasicom} $A$ is called \textit{quasi-commutative} if the conditions
{\rm(}iii{\rm)} and {\rm(}iv{\rm)} in Definition \ref{defpbwt} are
replaced by \begin{enumerate}
\item[\rm (iii')] for each $1\leq i\leq n$ and all $r\in R\ \backslash\ \{0\}$, there exists $c_{i,r}\in R\ \backslash\ \{0\}$ such that
\begin{equation}\label{rel1.quasi}
x_ir=c_{i,r}x_i;
\end{equation}
\item[\rm (iv')]for any $1\leq i,j\leq n$, there exists $c_{i,j}\in R\ \backslash\ \{0\}$ such that
\begin{equation}\label{rel2.quasi}
x_jx_i=c_{i,j}x_ix_j.
\end{equation}
\end{enumerate}
\item[\rm (b)]  $A$ is called \textit{bijective}, if $\sigma_i$ is bijective for each $\sigma_i\in \Sigma$, and $c_{i,j}$ is invertible for any $1\leq
i<j\leq n$.
\item[\rm (c)] If $\sigma_i={\rm id}_R$ for every $\sigma_i\in \Sigma$, we say that $A$ is a skew PBW extension of \textit{derivation type}.
\item[\rm (d)]  If $\delta_i = 0$ for every $\delta_i\in \Delta$, we say that $A$ is a skew PBW extension of \textit{endomorphism type}.
\item[\rm (e)]   Any element $r$ of $R$ such that $\sigma_i(r)=r$ and $\delta_i(r)=0$ for all $1\leq i\leq n$, will be called a \textit{constant}. $A$ is called \textit{constant} if every element of $R$ is constant.
\end{enumerate}
\end{definition}

Graded skew PBW extensions were defined by the second author in
\cite{Hec} in order to study the Koszul property  and Artin-Schelter
regular property for  skew PBW extensions.

\begin{proposition}[\cite{Hec}, Proposición 2.7] \label{prop24}
Let $R=\bigoplus_{m\geq 0}R_m$ be a $\mathbb{N}$-graded algebra and
let $A=\sigma(R)\langle x_1,\dots, x_n\rangle$ be a bijective skew
PBW extension of $R$ satisfying the fo\-llo\-wing two conditions:
\begin{enumerate}
\item[\rm (i)] $\sigma_i$ is a graded ring homomorphism and $\delta_i : R(-1) \to R$ is a graded $\sigma_i$-derivation for all $1\leq i  \leq n$,
where $\sigma_i$ and $\delta_i$ are as in Proposition \ref{prop23}.
\item[\rm (ii)]  $x_jx_i-c_{i,j}x_ix_j\in R_2+R_1x_1 +\cdots + R_1x_n$, as in (\ref{torcida4}) and $c_{i,j}\in R_0$.
\end{enumerate}
For $p\geq 0$, let $A_p$ the $K$-space generated by the set
\[\Bigl\{r_tx^{\alpha} \mid t+|\alpha|= p,\  r_t\in R_t \quad\text{and}\quad x^{\alpha}\in {\rm
Mon}(A)\Bigr\}.
\]
Then $A$ is a $\mathbb{N}$-graded algebra with graduation
\begin{equation}\label{eq.grad alg skew}
A=\bigoplus_{p\geq 0} A_p.
\end{equation}
\end{proposition}

\begin{definition}[\cite{Hec}, Definition 2.6]\label{defgrd} Let  $A=\sigma(R)\langle x_1,\dots, x_n\rangle$ be a bijective
skew PBW extension of an  $\mathbb{N}$-graded algebra
$R=\bigoplus_{m\geq 0}R_m$. We say that $A$ is a \emph{graded  skew
PBW extension} if $A$ satisfies the conditions (i) and (ii) in
Proposition \ref{prop24}.
\end{definition}

\begin{example}\label{ex.rig}
Let $A$ be a right double Ore extension of $K$, then $A$ is
isomorphic to the algebra $K \langle x_1,x_2 \rangle/(x_2x_1 -
p_{12}x_1x_2 -p_{11}x_1^2 - a_1x_1-a_2x_2-a_3)$, for some
$p_{12},p_{11},a_1,a_2,a_3 \in K$. We note that if $p_{12}\neq 0$
and $p_{11}=0$ then $A$ is a constant skew PBW extension, but $A$ is
not in general a graded skew PBW extension, nor a pre-commutative
skew PBW extension, nor a quasi-commutative skew PBW extension.
\end{example}

\begin{remark}\label{notagrd} Let $A = \sigma(R)\langle x_1, \dots, x_n \rangle$  be a graded skew PBW extension. Note that $R$ is connected
if and only if $A$ is connected.
\end{remark}

\section{Skew PBW extensions versus double Ore
extensions}\label{sec.versus}

Iterated Ore extension of bijective type (in \cite{Le} these
algebras are called iterated skew polynomial ring of bijective type)
are skew PBW extensions. Let $B= R[x_1; \sigma_1, \delta_1]$ be an
Ore extension of  $R$ and $C = B[x_2; \sigma_2,\delta_2]$ be an Ore
extension of $B$, in general $C$ is not a right double Ore extension
because $C$ might not have a relation of the form (\ref{R1}). These
double Ore extensions are presented as  iterated Ore extensions of
the form $R[x_1; \sigma_1, \delta_1][x_2;\sigma_2,\delta_2]$ in
\cite{Pau}. To best of our knowledge, there is no study of the
relations between skew PBW extensions and double Ore extensions.\

The Weyl algebra $A_2(K)$ is an iterated Ore extension, a skew PBW
extension, and a right double Ore extension. Indeed, since $A_2(K) =
K[t_1,t_2][x_1,\partial/\partial t_1][x_2,\partial/\partial t_2]$,
then this algebra is an iterated Ore extension of $K[t_1,t_2]$. Note
that, $x_ir = rx_i+
\partial r/\partial t_i$, $x_2x_1-x_1x_2= 0$, for any $r\in
K[t_1,t_2]$ and $1 \leq i \leq 2$. Thus, $A_2(K)$ is a skew PBW
extension of $K[t_1,t_2]$, with $\sigma_1=\sigma_2$ the identity map
in $K[t_1,t_2]$, $\delta_1= \partial/\partial t_1$ and
$\delta_2=\partial/\partial t_2$, where $\sigma_1$, $\sigma_2$,
$\delta_1$, $\delta_2$ are given as in Proposition \ref{prop23}.
$A_2(K)=\sigma(R)\langle x_1,x_2\rangle$ is a pre-commutative and
bijective skew PBW extension of derivation type. $A_2(K)$ is not a
graded skew PBW extension, nor a quasi-commutative skew PBW
extension. Note that $A_2(K)$ is a right double Ore extension of
$K[t_1,t_2]$, with $p_{12}=1$, $p_{11}=0$, $\tau_0=\tau_1=\tau_2=
0$, $\sigma_{11}$ and $\sigma_{22}$ the identity map in
$K[t_1,t_2]$, $\sigma_{12}=\sigma_{21}=0$, $\delta_1(r)=
\partial r/\partial t_1$ and $\delta_2(r)=\partial r/\partial t_2$, for all $r\in K[t_1,t_2]$;
where $p_{12}$, $p_{11}$, $\tau_0$, $\tau_1$, $\tau_2$ are given as
in (\ref{R1}) of Definition \ref{defore2} and $\sigma_{11}$,
$\sigma_{22}$, $\sigma_{12}$, $\sigma_{21}$, $\delta_1$, $\delta_2$
are given as in Equation (\ref{R2}). $A=R_P[x_1,x_2; \sigma,
\delta,\tau]$ is not a trimmed right double Ore extension since
$\delta_1(t_1)=1=\delta_2(t_2)$.\

The concept of double Ore extension (Definition \ref{defore2}) and
the skew PBW extension in two variables, $A =\sigma(R)\langle
x_1,x_2\rangle$, are similar because of the characteristics that
they have. The difference between these structures lies, in general
terms, in the following fact: in the
 right double Ore extension property (\ref{R1}), $x_2x_1= p_{12}x_1x_2 + p_{11}x_1^2 + \tau_1x_1 + \tau_2 x_2 +
\tau_0$, where $p_{12}, p_{11} \in K$ and $\tau_1, \tau_2, \tau_0
\in R,$ while in the property (\ref{torcida4}) of the definition of
skew PBW extension $x_2x_1 - c_{1,2}x_1x_2 \in R + Rx_1 + Rx_2$,
with $c_{1,2} \in R \setminus \{0\}$. If $p_{11} = 0$, it would be
necessary that $ 0\neq p_{12} = c_{1,2}\in K$, but this is not
always true. The following theorem gives the conditions for a double
Ore extension to be a skew PBW extension of an  algebra $R$.

\begin{theorem}\label{rel11}
Let $A = R_P[x_1,x_2,\sigma, \delta,\tau]$ be a right double Ore
extension of $R$ with $P = \{p_{12},0\}$, $p_{12}\neq 0$, $\sigma :=
\left(\begin{smallmatrix}
\sigma_{11} && 0 \\
0 && \sigma_{22} \\
\end{smallmatrix}\right)$,  where $\sigma_{11}(r)\neq 0$, $\sigma_{22}(r) \neq 0$, for all $r \in R \setminus \{0\}$ and  $\delta := \left(\begin{smallmatrix}
\delta_1  \\
\delta_2  \\
\end{smallmatrix}\right)$. Then $A$ is a skew PBW extension of $R$.
\end{theorem}

\begin{proof}
Since $A$ is a right double Ore extension of $R$, then $R\subseteq
A$ and $A$ is a left free $R$-module with
$\{x_1^{\alpha_1}x_2^{\alpha_2} \mid\alpha_1\geq 0,\alpha_2\geq
0\}$. Therefore the conditions {\rm (i)} and {\rm (ii)} of
Definition \ref{defpbwt} are satisfied. Furthermore,  $x_1r =
\sigma_{11}(r)x_1 + \sigma_{12}(r)x_2 + \delta_1(r)  =
\sigma_{11}(r)x_1 + \delta_1(r), $  for all $r \in R\setminus\{0\}$,
and $x_2r = \sigma_{21}(r)x_1+\sigma_{22}(r)x_2+\delta_2(r) =
\sigma_{22}(r)x_2+\delta_2(r),$ where $\sigma_{11}(r),\sigma_{22}(r)
\neq 0$. Assume $c_{1,r}=\sigma_{11}(r)$ and $c_{2,r}=
\sigma_{22}(r)$, then we have (\ref{torcida3}). Now by (\ref{R1})
$$\begin{array}{lll}
x_2x_1  & = & p_{12}x_1x_2 + p_{11}x_1^2 + \tau_1x_1 + \tau_2 x_2 + \tau_0  \\
        & = & p_{12}x_1x_2 + \tau_1x_1 + \tau_2 x_2 + \tau_0,
\end{array}$$
 where $0\neq p_{12} = c_{1,2} \in K $ and $\tau_1, \tau_2, \tau_0 \in
 R$. Also, $$\begin{array}{lll}
p_{12}^{-1}x_2x_1  & = & x_1x_2 + p_{12}^{-1}\tau_1x_1 + p_{12}^{-1}\tau_2 x_2 + p_{12}^{-1}\tau_0  \\
                   & = & x_1x_2 + \tau_1'x_1 +  \tau_1'x_2 + \tau_0',
\end{array}$$
where $0\neq p_{12}^{-1} = c_{2,1} \in K $ and $\tau_1', \tau_2',
\tau_0' \in R$, with which we obtain (\ref{torcida4}).\\
\end{proof}

\begin{theorem}\label{fff} Let $A= \sigma(R)\langle x_1,x_2\rangle$ be a connected graded skew PBW extension of an algebra
$R$. Then $A$ is a  connected graded double Ore extension of $R$.
\end{theorem}

\begin{proof}
Conditions (ia) and (iiia) of Definition \ref{defore2} are obtained
from conditions (i) and (ii) of Definition \ref{defpbwt}. From
(\ref{torcida3}) we have that for each $r\in R$ $x_1r=c_{1,r}x_1 +
t_{1,r}=\sigma_1(r)x_1+\delta_1(r)$ and $x_2r=c_{2,r}x_2 + t_{2,r}=
\sigma_2(r)x_2+\delta_2(r)$, with $c_{1,r}, c_{2,r},t_{1,r},
t_{2,r}\in R$ and $\sigma_1$, $\sigma_2$, $\delta_1$, $\delta_2$  as
in Proposition \ref{prop23}. Then we have that $x_1r+x_2r=c_{1,r}x_1
+ c_{2,r}x_2+(t_{1,r}+t_{2,r})$. Thus, $x_1R+x_2R\subseteq R_1x_1 +
Rx_2 + R$. So, (\ref{4ore2}) follows. Since $A$ is connected, by
Remark \ref{notagrd} we have that $R\bigoplus_{m\geq 0}R_m$ is
connected. So $R_0=K$. By Definition \ref{defgrd} and Proposition
\ref{prop24}-(ii) we have that $0 \neq c_{1,2} \in R_0 = K$. For
$p_{11} = 0$ and $p_{12} = c_{1,2}$ we have that $x_2x_1 =
c_{1,2}x_1x_2+\tau_1x_1+\tau_2x_2 + \tau_0$, where $\tau_1,\tau_2$
and $\tau_3 \in R$, i.e., $A$ is a double Ore extension of $R$. Now,
by Proposition \ref{prop24} and Definition \ref{defgrd} it's true
that the above relations are homogeneous. Thus, $A$ is a connected
graded right double Ore extension of $R$. By Proposition
\ref{prop2.1do}-(iii), $A= R_P[x_1,x_1;\sigma,\delta,\tau]$ is a
double Ore extension of $R$ where $P=\{c_{1,2},0\}$,
$$\sigma(r) := \left(\begin{matrix}
\sigma_{11}(r) && 0 \\
0 && \sigma_{22}(r) \\
\end{matrix}\right)=\left(\begin{matrix}
c_{1,r} && 0 \\
0 && c_{2,r} \\
\end{matrix}\right)=\left(\begin{matrix}
\sigma_1(r) && 0 \\
0 && \sigma_1(r) \\
\end{matrix}\right),$$  $$\delta(r) := \left(\begin{matrix}
t_{1,r}  \\
t_{2,r}  \\
\end{matrix}\right)= \left(\begin{matrix}
\delta_1(r)  \\
\delta_1(r) \\
\end{matrix}\right),$$  and $\tau =
\{\tau_0,\tau_1,\tau_2\}$, with $\tau_0 \in R_2$ and
$\tau_1$,$\tau_2 \in R_1$.
\end{proof}

\begin{corollary}\label{cor.traim} Let $A= \sigma(R)\langle x_1,x_2\rangle$ be a connected graded quasi-commutative skew PBW extension of
$R$. Then $A$ is a connected graded trimmed right double Ore
extension of $R$.
\end{corollary}
\begin{proof}
By Theorem \ref{fff} we have that $A$ is a connected graded double
Ore extension of $R$. Since $A$ is quasi-commutative, then  there
exists $c_{1,r}, c_{2,r},c_{1,2} \in R\ \backslash\ \{0\}$ such that
\begin{equation}
x_1r=c_{1,r}x_1,\quad   x_2r=c_{2,r}x_2 \text{ for all } r\in R\
\backslash\ \{0\}
\end{equation}
and
\begin{equation}
x_jx_i=c_{i,j}x_ix_j.
\end{equation}
Therefore, $\begin{array}{llllll} \delta(r)&=& \left(\begin{matrix}
\delta_{1}(r)  \\
\delta_{2}(r)  \\
\end{matrix}\right) &=& \left(\begin{matrix}
0 \\
0 \\
\end{matrix}\right),
\end{array}$
and $\tau_1=\tau_2=\tau_0=0.$ Thus, $A=R_P[x_1,x_2;\sigma]$ is a
trimmed right double Ore extension, where $P=\{c_{1,2}, 0\}$ and
$$\sigma(r) = \left(\begin{smallmatrix}
\sigma_{11}(r) && 0 \\
0 && \sigma_{22}(r) \\
\end{smallmatrix}\right)= \left(\begin{smallmatrix}
\sigma_{1}(r) && 0 \\
0 && \sigma_{2}(r) \\
\end{smallmatrix}\right),$$ with $\sigma_{1}$ and $\sigma_{2}$ as in
Proposition \ref{prop23}.
\end{proof}

\begin{example}\label{asejeml} The following algebras are connected graded
skew PBW extensions (see \cite{Hec,Hect}). By Theorem \ref{fff}
these algebras are also connected graded right double Ore
extensions.

\begin{enumerate}
\item \emph{Homogenized enveloping algebra}. Let $\mathcal{G}$ be a two dimensional Lie algebra over
$K$ with basis $\{x_1, x_2\}$ and $\mathcal{U}(\mathcal{G})$ its
enveloping algebra. The \emph{homogenized enveloping algebra} of
$\mathcal{G}$ is $\mathcal{A}(\mathcal{G}):= T(\mathcal{G}\oplus
Kz)/\langle R\rangle$, where $T(\mathcal{G}\oplus Kz)$ is the tensor
algebra, $z$ is a new variable, and $R$ is spanned by $\{z\otimes
x-x\otimes z\mid x\in \mathcal{G}\}\cup \{x\otimes y-y\otimes
x-[x,y]\otimes z\mid x,y\in \mathcal{G}\}$.
$\mathcal{A}(\mathcal{G})$ is not an iterated Ore extension. We get
that $\mathcal{A}(\mathcal{G})$ is a graded right double Ore
extension of $K[z]$.
\item \emph{Diffusion algebra}. The  diffusion algebra $\mathcal{D}$ generated by $\{D_1, D_2, x_1, x_2\}$ over $K$ with relations
$x_2x_1 = x_1x_2$, $x_iD_j = D_jx_i$, $1 \leq i, j \leq 2$; \quad
$c_{1,2}D_1D_2 - c_{2,1}D_2D_1 = x_2D_1 - x_1D_2$, $c_{1,2}$,
$c_{2,1}\in  K\setminus\{0\}$, is a graded non quasi-commutative
skew PBW extension of $K[x_1, x_2]$. Observe that $\mathcal{D}$ is
not an iterated Ore extension. $\mathcal{D}$ is a graded right
double Ore extension of $K[x_1, x_2]$.
\item \emph{Algebra of linear partial q-dilation operators}. For a
fixed $q \in K\setminus \{0\}$, the algebra of linear partial
$q$-dilation operators with polynomial coefficients is $$K[t_1,\dots
,t_n][H^{(q)}_1, H^{(q)}_2],\quad n \geq 2,$$ subject to the
relations: $t_jt_i = t_it_j$, $1 \leq i < j \leq n$;\quad $H^{(q)}_i
t_i = qt_iH^{(q)}_i$, $1 \leq i \leq 2$;\quad $H^{(q)}_j t_i =
t_iH^{(q)}_j$, $i \neq j$;\quad $H^{(q)}_j H^{(q)}_i =
H^{(q)}_iH^{(q)}_j$, $1 \leq i < j\leq 2$. This algebra is a
connected graded quasi-commutative skew PBW extension of
$K[t_1,\dots ,t_n]$. By Corollary \ref{cor.traim}, $K[t_1,\dots
,t_n][H^{(q)}_1, H^{(q)}_2]$ is a connected graded trimmed right
double Ore extension of $K[t_1,\dots,t_n]$.
\item \emph{Quantum plane}. This algebra is denoted by $\mathcal{O}_2(\lambda_{ji})$ and is
generated by $x_1,x_2$ subject to the relation: $x_2x_1
=\lambda_{21}x_1x_2$, with $\lambda_{21}\in K\setminus \{0\}$. This
algebra is a connected graded quasi-commutative skew PBW extension
of $K$. By Corollary \ref{cor.traim}, $\mathcal{O}_2(\lambda_{ji})$
is a connected graded trimmed right double Ore extension of $K$.
\end{enumerate}
\end{example}

Next we show that some right double Ore extensions are  graded skew
PBW extensions.

\begin{theorem}\label{rel112} Let $R= \bigoplus_{m\geq 0}R_m$ be an $\mathbb{N}$-graded algebra and let $A=R_P[x_1,x_2,\sigma,\delta,\tau]$
be a graded right double Ore extension of $R$. If $P=\{p_{12},0\}$,
$p_{12}\neq 0$ and $\sigma := \left(\begin{smallmatrix}
\sigma_{11} && 0 \\
0 && \sigma_{22} \\
\end{smallmatrix}\right)$, where $\sigma_{11}$, $\sigma_{22}$ are automorphism of $R$, then  $A$ is a graded skew PBW extension of $R$.
\end{theorem}

\begin{proof}
By Theorem \ref{rel11} and its proof, $A$ is a skew PBW extension of
$R$ with $c_{1,r}=\sigma_{11}(r)$, $c_{2,r}= \sigma_{22}(r)$,
$p_{12}^{-1} = c_{2,1}\in K\setminus \{0\}$ and $0\neq p_{12} =
c_{1,2}\in K\setminus \{0\}$. Since $\sigma_{11}$ and $\sigma_{22}$
are bijective, we have that $A$ is a bijective skew PBW extension.
Furthermore, all relations of $A=R_P[x_1,x_2,\sigma,\delta,\tau]$
are homogeneous since $A$ is a graded right double Ore extension of
$R$. Then, conditions (i) and (ii) in Proposition \ref{prop24} are
satisfied. Thus, $A$ is a graded skew PBW extension of $R$.
\end{proof}

\begin{example}\label{oreejm1}
Let $R=K[t]$. The double Ore extension of $R$, $A^1:=A^1(p,a,b,c)$,
is the algebra generated by $t, x_1, x_2$  defined by the relations
 $$\begin{array}{ll}
x_2x_1 & = px_1x_2 +\dfrac{bc}{1-b}(pb-1)tx_1+ at^2,\\
x_1t   & = btx_1, \\
x_2t    &= b^{-1}tx_2+ct^2,
 \end{array}$$
 where $a,b,c,p \in K$ and $p\neq 0$, $b\neq 0,1$ (see \cite[Example 4.1]{Zhang}). The homomorphism $\sigma$ is determined by
 $\sigma(t) = \left(\begin{smallmatrix}
bt && 0 \\
0 && b^{-1}t \\
\end{smallmatrix}\right)$. In this case $\sigma_{12}=\sigma_{21}=0$ and $\sigma_{11}$, $\sigma_{22}$ are algebra
automorphisms. The derivation is determined by  $\delta(t) =
\left(\begin{smallmatrix}
0 \\
ct^2 \\
\end{smallmatrix}\right)$. The parameter $P$ is $(p,0)$ and the tail is $\tau = \{\dfrac{bc}{1-b}(pb-1)t, 0,at^2\}$.
Then, from Theorem \ref{rel11} we have that $A^1$ is a skew PBW
extension of $K[t]$. Assuming that deg$(t)=$deg$(x_1)=$deg$(x_2)=1$,
we have that  $A^1$ is a connected graded skew PBW extension.
\end{example}

\begin{example}\label{oreejm3} The double Ore extension $A^3:=A^3(a)$ of $K[t]$ (see \cite[Example
4.1]{Zhang}) defined by the relations

$$\begin{array}{ll}
x_2x_1 & = x_1x_2,\\
x_1t   & = atx_1+tx_2, \\
x_2t    &= atx_2,
 \end{array}$$
 where  $a\in K$ and $a\neq 0$, is a skew PBW extension of $K[x_2]$. In fact, we have that
 $tx_2=a^{-1}x_2t$.  Hence
 $$\begin{array}{ll}
x_1x_2 & = x_2x_1,\\
tx_2    &= a^{-1}x_2t,\\
x_1t   & = atx_1+a^{-1}x_2t.
 \end{array}$$
Assuming that deg$(t)=$deg$(x_1)=$deg$(x_2)=1$, we have that  $A^3$
is a connected graded skew PBW extension.
\end{example}

By Theorem \ref{fff} we have that if $A= \sigma(R)\langle
x_1,x_2\rangle$ is a connected graded skew PBW extension of an
algebra $R$, then $A$ is a connected graded double Ore extension of
$R$. In the next theorem, we give necessary and sufficient
conditions for a connected graded  double Ore extensions of a
connected algebra $R$ to be a graded  skew PBW extension.

\begin{theorem}\label{relaciondoblesgrad} Let $A=R_P[x_1,x_2,\sigma]$ be a graded double Ore extension of a connected algebra $R$. Then  $A$ is a graded
 skew PBW extension of $R$ if and only if
$p_{12}\neq 0$, $p_{11}=0$, $\sigma_{11}$, $\sigma_{22}$ are
automorphisms of $R$ and $\sigma_{12}(r)=\sigma_{21}(r)=0$, for all
$r\in R$.
\end{theorem}
\begin{proof}
Let $A=R_P[x_1,x_2,\sigma]$ be a graded  double Ore extension of a
connected algebra $R$. Then from (\ref{R2}),
\begin{equation}\label{R2grad}
x_1r=\sigma_{11}(r)x_1+\sigma_{12}(r)x_2+\delta_1(r)\quad \text{ and
}\quad x_2r=\sigma_{21}(r)x_1+\sigma_{22}(r)x_2+\delta_2(r),
\end{equation}
for all $r \in R$.\\
From (\ref{R1}),
\begin{equation}\label{R1grad}
x_2x_1 = p_{12}x_1x_2 + p_{11}x_1^2 + \tau_1x_1 + \tau_2x_2 +
\tau_0,
\end{equation}
where $p_{12}, p_{11}\in K$ and $\tau_1, \tau_2, \tau_0 \in R$.\\
Since $A$ is graded, the relations (\ref{R2grad}) and (\ref{R1grad})
are homogeneous. If $A$ is a graded skew PBW extension of $R$, then
  equations (\ref{R2grad}) and (\ref{R1grad}) become
\begin{equation}\label{R2quasi}
x_1r=\sigma_{11}(r)x_1+\delta_1(r)=
\sigma_{1}(r)x_1+\delta_1(r)\quad \text{ and } \quad
x_2r=\sigma_{22}(r)x_2+\delta_2(r)=\sigma_{2}(r)x_2+\delta_2(r),
\end{equation}
where $\sigma_1$, $\sigma_2$, $\delta_1$, $\delta_2$ are as in Proposition \ref{prop23};\\
\begin{equation}\label{R1quasi} x_2x_1 = p_{12}x_1x_2 + \tau_1x_1 + \tau_2x_2 +
\tau_0,
\end{equation}
Therefore, $p_{12}\neq 0$, $p_{11}=0$, $\sigma_{11}=\sigma_1$,
$\sigma_{22}=\sigma_2$
are automorphism of $R$ and $\sigma_{12}(r)=\sigma_{21}(r)=0$, for all $r\in R$.\\
For the converse implication, suppose that $p_{12}\neq 0$,
$p_{11}=0$, $\sigma_{11}$, $\sigma_{22}$ are automorphism of $R$ and
$\sigma_{12}(r)=\sigma_{21}(r)=0$, for all $r\in R$. Then by Theorem
\ref{rel112}, $A$ is a graded skew PBW extension of $R$.
\end{proof}

By Corollary \ref{cor.traim}, connected graded quasi-commutative
skew PBW extensions of $R$ are connected graded trimmed right double
Ore extensions of $R$. In the next theorem, we give necessary and
sufficient conditions for a connected graded trimmed right double
Ore extension to be a connected graded quasi-commutative skew PBW
extension.

\begin{theorem}\label{teo.traim} Let $A=R_P[x_1,x_2,\sigma]$ be a graded trimmed right double Ore extension
of a connected algebra $R$. Then  $A$ is a graded quasi-commutative
skew PBW extension of $R$ if and only if $p_{12}\neq 0$, $p_{11}=0$,
$\sigma_{11}$, $\sigma_{22}$ are automorphism of $R$ and
$\sigma_{12}(r)=\sigma_{21}(r)=0$, for all $r\in R$.
\end{theorem}

\begin{proof}
Let $A=R_P[x_1,x_2,\sigma]$ be a graded trimmed right double
extension of a connected algebra $R$. Then from (\ref{R2}),
\begin{equation}\label{R2trimed} x_1r=\sigma_{11}(r)x_1+\sigma_{12}(r)x_2\quad\text{ and
}\quad x_2r=\sigma_{21}(r)x_1+\sigma_{22}(r)x_2,
\end{equation}
for all $r \in R$.\\
From (\ref{R1}),
\begin{equation}\label{R1trim} x_2x_1 = p_{12}x_1x_2 + p_{11}x_1^2,
\end{equation}
where $p_{12}, p_{11}\in K$.\\
If $A$ is a graded quasi-commutative skew PBW extension of $R$, then
  equations (\ref{R2trimed}) and (\ref{R1trim}) become
\begin{equation}\label{R2quasi}
x_1r= \sigma_{11}(r)=\sigma_1(r)x_1 \quad\text{ and }\quad x_2r=
\sigma_{22}(r)x_2=\sigma_2(r),
\end{equation}
where $\sigma_1$ and $\sigma_2$ are as in Proposition \ref{prop23};\\
\begin{equation}\label{R1quasi} x_2x_1 = p_{12}x_1x_2.
\end{equation}
Therefore, $p_{12}\neq 0$, $p_{11}=0$, $\sigma_{11}$, $\sigma_{22}$
are automorphism of $R$ and $\sigma_{12}(r)=\sigma_{21}(r)=0$, for all $r\in R$.\\
For the converse implication, suppose that $p_{12}\neq 0$,
$p_{11}=0$, $\sigma_{11}$, $\sigma_{22}$ are automorphism of $R$ and
$\sigma_{12}(r)=\sigma_{21}(r)=0$, for all $r\in R$. By Theorem
\ref{rel112}, $A$ is a graded skew PBW extension of $R$. Now
equations (\ref{R2trimed}) and (\ref{R1trim}) become
\begin{equation}\label{R2conv} \left(\begin{smallmatrix}
x_1r \\
x_2r \\
\end{smallmatrix}\right)= \left(\begin{smallmatrix}
\sigma_{11}(r) && 0 \\
0 && \sigma_{22}(r) \\
\end{smallmatrix}\right)\left(\begin{smallmatrix}
x_1 \\
x_2 \\
\end{smallmatrix}\right),
\end{equation}
for all $r \in R$; and\\
\begin{equation}\label{R1conv} x_2x_1 = p_{12}x_1x_2.
\end{equation}
From (\ref{R2conv}) we obtain
\begin{equation}\label{R3conv}
x_1r=\sigma_{11}(r)x_1 \quad\text{ and }\quad
x_2r=\sigma_{22}(r)x_2.
\end{equation}
From (\ref{R1conv}) and (\ref{R3conv}) we have that $A$ is
quasi-commutative.
\end{proof}

In \cite{Hect} it is shown that  graded quasi-commutative skew PBW
extensions of Artin-Schelter regular algebras are Artin-Schelter
regular, and that graded skew PBW extensions of  Auslander regular
algebras are Artin-Schelter regular. Since Auslander regular
algebras are Artin-Schelter regular, in the next theorem we
generalize
 these results for skew PBW extensions in two variables.

\begin{theorem}\label{result}
If $R$ is an Artin-Schelter regular algebra and $A =
\sigma(R)\langle x_1,x_2 \rangle$ is a graded skew PBW extension of
$R$, then $A$ is  Artin-Schelter regular and $gld(A)=gld(R)+2$.
\end{theorem}

\begin{proof}
If $R$ is Artin-Schelter regular then it is connected. By Remark
\ref{notagrd}, $A$ is connected. Thus, by Theorem \ref{fff}, $A$
connected graded  double Ore extension
of $R$. By Theorem \ref{Teorema0.2}, $A$ is Artin-Shelter regular.\\
\end{proof}

\begin{definition}\label{def1.1} A graded algebra $B$ is called \emph{skew Calabi-Yau} of dimension $d$ if
\begin{enumerate}

\item[(i)] $B$ is homologically smooth, i.e, as an $B^e$-module, $B$ has a   projective re\-so\-lu\-tion with finite length  such that each term in the projective resolution is finitely generated.
\item[(ii)]There exists an  algebra automorphism  $\nu$ of $B$ such that \[Ext^i_{B^e} (B,B^e) \cong \left\{
                                         \begin{array}{ll}
                                           0, & i\neq d; \\
                                           B^{\nu}(l), & i= d.
                                         \end{array}
                                       \right.\]
 as $B^e$ -modules, for some integer $l$.
\end{enumerate}
\end{definition}

Reyes,  Rogalski and Zhang in \cite{Reyes} proved that for a
connected graded algebra $B$, $B$ is skew Calabi-Yau  if and only if
it is Artin-Schelter regular. Using this fact and Theorem
\ref{result}, we have the following result.

\begin{corollary}\label{resultCalab} Let  $R$ be a connected algebra. If $R$ is skew Calabi-Yau of dimension $d$, then every graded skew PBW extension
$A = \sigma(R)\langle x_1,x_2 \rangle$  is skew Calabi-Yau of
dimension $d+2$.
\end{corollary}

\begin{example}\label{ex.ArtinCal} With the above results we have that Examples \ref{ex.rig}, \ref{asejeml}, \ref{oreejm1},
\ref{oreejm3} are Artin-Schelter regular and skew Calabi-Yau
algebras.
\end{example}

\end{document}